\documentclass[final]{amsart}
\usepackage{microtype}
\usepackage[foot]{amsaddr}

\usepackage[utf8]{inputenc}
\usepackage{amsmath}
\usepackage{amsfonts}
\usepackage{amssymb}
\usepackage{amsthm}
\usepackage{graphicx}
\usepackage{color}
\usepackage[colorlinks= false, urlcolor=black, citecolor=black, breaklinks=true]{hyperref}
\usepackage{tikz}
\usetikzlibrary{positioning}
\usetikzlibrary{calc}
\usepackage{bm}
\usepackage{isomath}
\usepackage{framed}
\usepackage{algorithm}
\usepackage{algpseudocode}
\usepackage{cite}

\newcommand{\remove}[1] {}

\newtheorem{definition}{Definition}
\newtheorem{lemma}{Lemma}

\newtheorem{theorem}{Theorem}

\newtheorem{claim}{Claim}

\newtheorem*{summary*}{Summary of results}
\newtheorem{example}{Example}

\newtheorem{innercustomthm}{Theorem}
\newenvironment{customthm}[1]
  {\renewcommand\theinnercustomthm{#1}\innercustomthm}
  {\endinnercustomthm}

\newcommand{\eat}[1]{}
\newcommand{\m}{\{1,\ldots,m\}}
\newcommand{\e}{\epsilon}
\newcommand{\N}{\mathbb{N}}
\newcommand{\F}{\mathcal{F}}

\newcommand{\E}{\bar{E}}
\newcommand{\X}{\mathcal{X}}
\newcommand{\x}{\chi}
\newcommand{\MAlg}{\textsc{M-Algorithm}}
\newcommand{\MSR}{\textsc{MaxSetRes}}
\newcommand{\Val}{\textsc{ValAlg}}
\newcommand{\MV}{\textsc{MultiSetVal}}
\newcommand{\p}{\bar{p}}

\newcommand{\MSV}{\textsc{MaxSetVal}}
\newcommand{\bt}{\mathbf{t}}

\newcommand{\Q}{\mathbf{Q}}
\newcommand{\R}{\mathbf{R}}
\newcommand{\bx}{\mathbf{x}}
\newcommand{\Btree}{\mathcal{B}}

\newcommand{\bn}{\mathbf{n}}
\newcommand{\1}{\mathbf{1}}
\newcommand{\G}{\Gamma}
\newcommand{\rP}{\rm{P}}
\newcommand{\hP}{\hat{\rm{P}}}
\newcommand{\tP}{\tilde{\rm{P}}}
\newcommand{\sco}{\text{sc}}
\newcommand{\z}{\bar{z}}
\newcommand{\I}{\mathcal{I}}
\newcommand{\bz}{\mathbf{z}}

\usepackage{varwidth}

\title[Directed Lov\'asz Local Lemma and Shearer's Lemma]{Directed Lov\'asz Local Lemma and Shearer's Lemma$^*$\footnote{$^*$T\lowercase{his is a revised version of \cite{kirousis2020directed} with a correct statement and proof of Theorem 1a.}}}

\author[L. Kirousis]{Lefteris Kirousis$^{1,3}$}
\thanks{$^3$Research partially supported  by TIN2017-86727-C2-1-R, GRAMM}

\author[J. Livieratos]{John Livieratos$^1$}

\author[K. I. Psaromiligkos]{Kostas I. Psaromiligkos$^{2,4}$}
\thanks{$^4$Research carried out while  an undergraduate student at the Department of Mathematics of the National and Kapodistrian University of Athens.}

\address{$^1$Department of Mathematics, National and Kapodistrian University of Athens}
\address{$^2$Department of Mathematics, University of Chicago. Onassis Scholar}

\email{\{lkirousis,jlivier89\}@math.uoa.gr, kostaspsa@math.uchicago.edu}

\begin{document}
\begin{abstract}
Moser and Tardos (2010) gave an algorithmic proof of the lopsided Lov\'asz local lemma (LLL) in the variable framework, where each of the undesirable events is assumed to depend on a subset of a collection of independent random variables. For the proof, they define a notion of a lopsided dependency  between the events suitable for this framework. In this work, we strengthen this notion, defining a novel \emph{directed} notion of dependency and prove the LLL for the corresponding graph. We   show that this graph can be strictly sparser (thus the sufficient condition for the LLL weaker) compared with graphs that correspond to  other extant lopsided versions of dependency. Thus, in a sense,  we address  the  problem ``find other simple local conditions for the constraints (in the variable framework) that advantageously translate to some abstract lopsided condition" posed  by Szegedy (2013). We also give an example where   our notion of dependency graph  gives  better results than the classical Shearer lemma. Finally, we prove  Shearer's lemma for the dependency  graph we define. For the  proofs, we perform a direct   probabilistic analysis that yields an exponentially small upper bound for  the probability  of the algorithm that searches for the desired assignment to the variables not to return a correct answer within $n$ steps. In contrast, the method of proof that became known as the entropic method, gives an estimate of only the expectation of the number of steps until the algorithm returns a correct answer, unless the probabilities are tinkered with. 
\end{abstract}

\maketitle

\section{Inroduction}\label{sec:intro}
The Lov\'asz Local Lemma (LLL) was originally stated and proved in 1975 by Erd\H{o}s and Lov\'asz \cite{erdoslovasz1975}. Its original {\em symmetric} form states that given events $E_1, \ldots, E_m$ in a common probability space, if every event depends  on at most $d$ others, and if the probabilities of all are bounded by $1/(4d)$, then $$Pr\Bigg[\bigwedge_{j=1}^m \bar{E}_j\Bigg] >0,$$ and therefore there exists at least one point in the space where none of the events occurs ($\bar{E}$ denotes the complement of $E$).

The {\em asymmetric} version entails  an undirected {\em dependency graph}, i.e. a graph with vertices $j=1, \ldots, m$ corresponding to the events $E_1, \ldots, E_m$  so that for all $j$, $E_j$ is mutually independent from the set of  events corresponding to vertices  not connected with $j$. The condition that in this case guarantees  that $Pr[\bigwedge_{j=1}^m \bar{E}_j] >0$ (and therefore that there exists at least one point where none of the events occurs) is:

\begin{equation}\label{eq:asymmetric}\tag{\sf{Asym}}\text{for every }E_j\text{ there is a }\x_j\in(0,1)\text{ such that }Pr[E_j]\leq\x_j\prod_{i\in N_j}(1-\x_i),\end{equation}
 where $N_j$ is the neighborhood of vertex $j$ in the dependency graph.

Improvements can be obtained by considering, possibly directed, sparser graphs than the dependency graph, that correspond to stronger  notions of dependency. For a classic example, the {\em lopsided }version (LLLL) by Erd\H{o}s and Spencer~\cite{erdosspencer1991} entails a {\em directed} graph with vertices corresponding to the events  such that for all $E_j$ and for all $I\subseteq\{1,\ldots,m\}\setminus(\G_j\cup\{j\})$ we have that \begin{equation}\label{eq:lopsidependency}\Pr\Bigg[E_j \mid \bigcap_{i \in I} \bar{E}_i\Bigg] \leq \Pr\left[E_j\right],\end{equation}
where $\G_j$ is the set of vertices connected with $j$ with an edge originating from $j$. Such graphs are known as {\em lopsidependency graphs}. The sufficient condition in this case that guarantees that the undesirable events can be avoided reads: 
\begin{equation}\label{eq:LLLL}\tag{\sf{Lop}}\text{for every }E_j\text{ there is a }\x_j\in(0,1)\text{ such that }Pr[E_j]\leq\x_j\prod_{i\in \G_j}(1-\x_i).\end{equation}
With respect to ordinary (not lopsided) dependency graphs, a sufficient {\em but also necessary} condition to  avoid all events was given by Shearer \cite{shearer1985problem}. It reads:

\begin{equation}\label{eq:shearer}\tag{\sf{Shear}} \text{For all } I\in I(G), \ q_I(G,\p):=\sum_{J\in I(G): I\subseteq J}(-1)^{|J\setminus I|}\prod_{j\in J}p_j>0,\end{equation} 
where $I(G)$ is the set of \emph{independent sets} of $G$ and $\p=(p_1,\ldots,p_m)$ is the vector of probabilities of the events.

By considering other graphs, and the corresponding to them Condition \eqref{eq:shearer}, variants of the Shearer lemma are obtained. These variants  are in  general  only sufficient, however they apply to sparser  dependency  graphs. For example, by proving the sufficiency of    Condition \eqref{eq:shearer} when  applied to  the  lopsidependency graph of Erd\H{o}s and Spencer \cite{erdosspencer1991}, we get  Shearer's lemma for lopsidependency graphs (actually, for  Shearer's lemma in this case it is the underlying undirected graph of  the lopsidependency graph that is considered). 

Apart from existence, research has been focused also in ``efficiently" finding a point in the probability space such that no undesirable event occurs. After several partially successful attempts that expand over more than three decades,  Moser~\cite{moser2009} in 2009, initially only for the symmetric LLL,  gave an extremely simple randomized  algorithm that if and when it stops, it certainly produces a point where all events are avoided. Soon after Moser and Tardos \cite{mosertardos2010} expanded this approach to  more general versions including the lopsided one. The algorithms were given  in the variable framework, where the space is assumed to be  the product space  of independent random variables $X_1, \ldots, X_l$,  and each event is assumed to depend on a subset of them, called its {\em scope}. Their algorithm just samples iteratively the variables of occurring events, until all the events cease to occur.  For the analysis, they estimate the \emph{expectation} of the number of times each event will be resampled in a given execution of the algorithm by counting ``tree-like'' structures they call \emph{witness trees} and by estimating the probability that such a tree occurs in the log of the algorithm's execution. This approach became known as the \emph{``entropy compression''} method (see \cite{tao2009} for a short  exposition).
For the proof of the lopsided LLL, Moser and Tardos \cite{mosertardos2010} defined an {\em undirected} lopsidependency graph suitable for the variable framework. 

 In the variable framework, Harris \cite{harris2015lopsidependency} gave a weaker version for the \ref{eq:LLLL} condition, entailing the notion of \emph{orderability}, which takes advantage of the way events are related based on the different values of the variables they depend on. He works with  the Moser-Tardos notion of lopsidependency graph and he proves that  his weaker sufficient condition  can yield  stronger results than the classical Shearer's lemma.  Also, recently He et al. \cite{he2017variable} gave a necessary and sufficient condition for LLL in the variable framework, but for the dependency graph where two events are connected if their scopes share at least one variable. We focus, on the contrary, in working with sparser graphs.

There are numerous applications of the algorithmic versions of both LLL and its lopsided version, even for problems that do not originate from purely combinatorial issues. For example, for the non-lopsided versions let us mention the problem of  \emph{covering arrays}, a problem closely related with software and hardware interaction testing. The objective is to find the \emph{minimum} number $N$, expressed as a function $k,t, v$, such that there exists an  $N\times k$ array $A$, with elements taken from a set $\Sigma$ of cardinality $v\geq 2$, so that every $N\times t$ sub-array of $A$   contains  as one of its rows every  element $x\in\Sigma^t$. Sarkar and Colbourn~\cite{sarkar2017upper} improve on known upper bounds for $N$, by using LLL. Notably,  they also provide an algorithm that constructs an $N\times k$ array with the above properties, by using a variant of the Moser-Tardos algorithm~\cite{mosertardos2010}.
As for the lopsided version let us mention, e.g., the work of  Harris and Srinivasan~\cite{harris2014constructive} who apply it  in the setting of \emph{permutations}, where the undesirable events are defined over permutations $\pi_k$ of $\{1,\ldots,n_k\}$, $k=1,\ldots,N$. 

The lopsided LLL was generalized to the   framework of arbitrary probability spaces  by Harvey and Vondr\'{a}k \cite{harvey2015algorithmic}, by means of a machinery that they called ``resampling oracles". They introduced, in the framework of arbitrary probability spaces, {\em directed} lopsidependency graphs they called {\em lopsided association} graphs. They proved that a graph is a lopsided association graph if and only if it is a graph along the edges of which  resampling oracles can be applied. In the generalized framework and based on their lopsided association condition, they algorithmically proved LLLL. In the same framework, they also proved Shearer's lemma (in this respect, see also the work of Kolipaka and Szegedy \cite{kolipakaszegedy2011}). Finally, Achlioptas and Iliopoulos \cite{achlioptas2016random}  have introduced a powerful abstraction for algorithmic LLL, which is inherently directed and they prove the lopsided LLL in this framework. 

Let us mention here that Szegedy \cite{szegedy2013} gives   a comprehensive survey of the LLL, that contains many of  the algorithmic results.

\paragraph{Our results.} We work exclusively in the variable framework. First, we define a novel relation of \emph{directed} dependen\-cy that we call $d$-\emph{de\-pend\-ency}, which is stronger than that of Moser and Tardos \cite{mosertardos2010}. We also show  that this relation  may generate  a strictly sparser dependency graph than other extant ones (and so it leads to weaker sufficient conditions for LLL).

We then algorithmically prove that the {\sf Lop} condition  suffices to avoid all events when applied to the graph defined by our notion of $d$-dependency. Thus, in a sense we address  the  problem ``find other simple local conditions for the constraints (in the variable framework) that advantageously translate to some abstract lopsided condition" posed  by Szegedy \cite{szegedy2013}. 

Our approach is based on Moser's original algorithm \cite{moser2009}, which, upon resampling an event, checks its neighborhood for other occurring events. Like in Giotis et al. \cite{giotiskirousis2015},   we use a witness structure (forest) to depict the execution of our algorithms that, in contrast with those of the ``Moser-Tardos-like'' proofs, grows ``forward in time'', meaning that it is  constructed as an execution moves on.  Taking advantage of this structure, we express the probability that the algorithm executes for at least $n$ rounds by a \emph{recurrence relation}. We subsequently solve this recurrence  by specialized analytical means, and prove that it diminishes exponentially fast in $n$. Specifically, we employ the result of Bender and Richmond \cite{bender1998multivariate} on the multivariable Lagrange inversion formula. A positive aspect of this approach is that it provides an exponentially small bound for  the probability of the algorithm to last for at least $n$ steps (including to run intermittently) before it returns the desired result,  in contrast to the entropic method  that estimates the expected time  of the algorithm to return a correct answer. We also note that, in contrast to  Harvey and Vondr\'{a}k \cite{harvey2015algorithmic}, our proof for the directed  LLL is independent of the one for Shearer's lemma. 

Finally, although we show that our notion of dependency can give stronger results than the classical Shearer's lemma,   we use our forward approach  to prove this lemma for the  $d$-dependency graph. 
 An algorithmic proof for Shearer's lemma for the ordinary dependency graph was first provided by Kolipaka and Szegedy \cite{kolipakaszegedy2011}, who actually gave a proof for the general case of arbitrary  probability spaces. The latter result was   strengthened   by Harvey and Vondr\'{a}k \cite{harvey2015algorithmic} for their notion of association graphs, again for general probability spaces.   Our result is for  the variable framework, but for the possibly sparser graph of $d$-dependency. Also,    we give again a direct computation of an exponentially small upper bound to the probability of the algorithm to last for at least $n$ steps. 
To carry out the computations in our forward approach,   we employ Gelfand's formula for the spectral radius of a matrix (see \cite{horn1990matrix}). 

Note that, as is the case with all extant algorithmic approaches to the LLL, both the number of events $m$ and the number of random variables $l$ are assumed to be constants. Complexity considerations are made with respect to the number of steps the algorithms last. 

\section{$d$-Dependency and a weak version of first result} \label{sec:d-depend}
For everything that follows, we assume that $Pr[E_j]<1$, $j=1,...,m$, lest there is no way to avoid all the events.

We begin by defining the following \emph{asymmetric} relation between two events.  

\begin{definition} \label{def:dd}
Given events $E_i, E_j$, we say that $E_i$ is \emph{$d$-dependent} on $E_j$ if:
\begin{enumerate}
\item there exists an assignment $\alpha$ to the random variables under which $E_j$ occurs  and $E_i$  does not, and
\item the values of the variables in sc($E_j$), the scope of $E_j$,  can be changed so that $E_i$ occurs and $E_j$ ceases occurring.
\end{enumerate}
\end{definition}
Intuitively, $E_i$ is $d$-dependent on $E_j$ if it is possible that some \emph{successful} attempt to avoid the occurrence of $E_j$ may end up with $E_i$ occurring, although initially it did not.

The binary relation of $d$-dependency defines a simple (no loops or multiple edges) \emph{directed} graph $G=(V,E)$, the \emph{$d$-dependency graph} of events $E_1,...,E_m$, where $V=\{1,...,m\}$ and
$E=\{(j,i) \mid E_i \mbox{ is  $d$-dependent on } E_j\}.$ Trivially, this graph is \emph{sparser} than the usual dependency graph in the variable framework, where there is an edge between events with \emph{intersecting scopes}.

For $i=1,...,m$, let $\G_j$ be the outwards neighborhood of the event $E_j$ in the $d$-dependency graph, i.e. $\G_j=\{i \mid E_i \mbox{ is $d$-dependent on } E_j\}.$
The notion of $d$-dependency was inspired by the following {\em symmetric} relation of Moser and Tardos \cite{mosertardos2010}, which will sometimes be referred to as {\em MT-dependency}:

\begin{definition}[Moser and Tardos \cite{mosertardos2010}] \label{def:lops}
Let $E_i, E_j$ be events, $i,j \in \{1,...,m\}$. We say that $E_i,E_j$ are lopsidependent if there exist two assignments $\alpha,\beta$, that differ only on variables in $\text{sc}(E_i)\cap \text{sc}(E_j)$, such that:
\begin{enumerate}
\item $\alpha$ makes $E_i$ occur and $\beta$ makes $E_j$ occur and
\item either $\E_i$ occurs under $\beta$ or $\E_j$ occurs under $\alpha$.
\end{enumerate}
\end{definition}
Moser and Tardos \cite{mosertardos2010} gave an algorithmic proof that if the condition \eqref{eq:asymmetric}  holds for an (undirected) dependency graph with respect to the notion of Definition \ref{def:lops}, then the undesirable events can be avoided.

The following claim is straightforward:
\begin{claim}\label{claim:compare}
If $E_i$ is $d$-dependent on $E_j$, in the sense of Definition \ref{def:dd}, then $E_i$ and $E_j$
 are MT-dependent, in the sense of Definition \ref{def:lops}.  
\end{claim}
\begin{proof}
Suppose that under $\alpha=(a_1,\ldots,a_l)$, $\E_i$ and $E_j$ occur and that we can change the values of the variables in sc($E_j$) to get and assignment $\beta=(b_1,\ldots,b_l)$ under which $E_i$ and $\E_j$ occur. Let now $\gamma=(c_1,\ldots,c_l)$ be such that:
\begin{equation}
    c_i=\begin{cases}
    b_i, & \text{if } X_i\in\text{sc}(E_i)\cap\text{sc}(E_j) \\
         a_i, & \text{else,}
    \end{cases}
\end{equation}
for $i=1,\ldots,l$.

Since $E_i$ is not affected by changes in $\text{sc}(E_j)\setminus\text{sc}(E_i)$, $E_i$ occurs under $\gamma$ and thus $E_i$, $E_j$ are lopsidependent.
\end{proof}
A weak version  of our result (given for comparison with other extant ones) reads: 

\begin{theorem}[Directed Lov{\'a}sz local lemma]\label{thm:dalll}
Suppose that there exist numbers $\x_1,\x_2,\ldots,\x_m\in(0,1)$, such that \begin{equation}\label{dirlop}\tag{\sf{DirLop}}\Pr(E_j)\leq \x_j \prod_{i\in \G_j} (1-\x_i),\end{equation}
for all $j\in\{1,\ldots,m\}$, where $\G_j$ denotes the neighborhood of $E_j$ in the $d$-dependency graph. Then, $$\Pr\Bigg[\bigwedge_{j=1}^m\bar{E}_j\Bigg]>0.$$
\end{theorem}
%Note that  \eqref{dirlop} differs from \eqref{eq:asymmetric}  and \eqref{eq:LLLL} only with respect to the graphs they are applied. 

Actually, we prove below  an algorithmic version  (Theorem \ref{thm:algalll}) of the existential Theorem~\ref{thm:dalll}, where we  give exponentially small estimates of the probability of the algorithm not producing the desired results within $n$ steps.

In the following example, we show that the $d$-dependency graph can be strictly sparser than other  dependency graphs that have been used in the literature.

\begin{example}\label{ex:compare}
Suppose we have $n\geq 3$ independent \emph{Bernoulli trials} $X_1,X_2, \ldots, X_n$, where $X_i=1$ denotes the event that the $i$-th such trial is \emph{successful}, $i=1,\ldots,n$. Consider also the $n$ ``undesirable events'':
$$E_j=\{X_j=1 \vee X_{j+1}=1\},$$ where $X_{n+1}=X_1$ and assume also that each of the Bernoulli trials succeeds with probability $p\in[0,1)$. Thus: $$Pr[E_j]=p+(1-p)p=2p-p^2.$$ We begin by showing that for any two \emph{distinct} $E_i$, $E_j$, neither one of them is $d$-dependent on the other. Without loss of generality, let $i=1$ and $j=2$.

Since $\text{sc}(E_1)\cup\text{sc}(E_2)=\{X_1,X_2,X_3\}$, both $E_1$ and $E_2$ are affected only by the first three coordinates of an assignment of values. We will thus restrict the assignmets to those coordinates.

Suppose $E_1$ and $\E_2$ occur under an assignment $\alpha$. Then, $\alpha=(1,0,0)$ and there is no way to change the first two coordinates in order for $\E_1$ and $E_2$ to occur. Thus $E_2$ is not $d$-dependent on $E_1$. Furthermore, for $\E_1$ and $E_2$ to occur under an assignment $\beta$, $\beta=(0,0,1)$ and there is no way to change the last two coordinates of $\beta$ in order for $E_1$ and $\E_2$ to occur. Thus $E_1$ is not $d$-dependent on $E_2$.

We can analogously prove the same things for all pairs of $E_j,E_{j+1}$, $j=1,\ldots,n$, where $E_{n+1}:=E_1$. Furthermore, it is easy to see that for any $i,j\in\{1,\ldots,n\}$: $i<j$ and $j\neq i+1$, neither $E_j$ is $d$-dependent on $E_i$ nor vice versa, since $E_i,E_j$ have no common variables they depend on. Thus, the $d$-dependency graph of the events has no edges and it is trivial to observe that we can avoid all the events if and only if $p<1$.

On the other hand, consider assignments $\gamma=(1,0,0)$ and $\delta=(1,1,0)$. Under $\gamma$, $E_1,\E_2$ occur, under $\delta$ $E_2$ occurs and the assignments differ only on $X_2\in\text{sc}(E_1)\cap\text{sc}(E_2)$. By Definition \ref{def:lops}, $E_1$ and $E_2$ are lopsidependent.

Given the above, it is not difficult to see that the underlying undirected graph of the  lopsidependency graph defined by the dependency relation of Definition \ref{def:lops}, is the cycle $C_n$ on $n$ vertices.

Interestingly, by interpreting Harvey and Vondr\'{a}k's \cite{harvey2015algorithmic} definition of resampling oracles in the variable setting as the resampling of the variables in the scope of an event, we get a \emph{directed graph}, whose underlying graph  is again $C_n$.  The same is true for the directed ``potential causality graph'' of Achliptas and Illiopoulos~\cite{achlioptas2016random}, where the flaws correspond to events and where we interpret an arc $f\rightarrow g$ between flaws $f,g$, again in the variable framework, as being able to obtain flaw $g$ by resampling the variables in the scope of flaw $f$.

In the sequel, we assume $n=3$ in order to simplify the example. The corresponding graphs are given in Figures \ref{fig:dd}--\ref{fig:hvai}.\newpage 

\begin{figure}
\begin{center}
		\begin{tikzpicture}[node/.style={circle,draw=black!100,fill=white!20,minimum size=15pt,inner sep=0pt},nonode/.style={circle}]
		\node[node] (d3) {$3$};
		\node[nonode] (d0) [above= of d3] {};
		\node[node] (d1) [left=0.5cm of d0] {$1$};
		\node[node] (d2) [right=0.5cm of d0] {$2$};
		\end{tikzpicture}
		\vspace{0.2cm}
\caption{$d$-dependency graph, Definition~\ref{def:dd}}
\label{fig:dd}
		\end{center}
\end{figure}
\begin{figure}
   \begin{center}
		\begin{tikzpicture}[node/.style={circle,draw=black!100,fill=white!20,minimum size=15pt,inner sep=0pt},nonode/.style={circle}]
		\node[node] (d3) {$3$};
		\node[nonode] (d0) [above= of d3] {};
		\node[node] (d1) [left=0.5cm of d0] {$1$};
		\node[node] (d2) [right=0.5cm of d0] {$2$};
		\draw (d1) -- (d2);
		\draw (d2) -- (d3);
		\draw (d3) -- (d1);
		\end{tikzpicture}
		\vspace{0.2cm}
    \caption{MT-dependency graph, Definition \ref{def:lops}}
    \label{fig:lops}
    \end{center}
\end{figure}
\begin{figure}
\begin{center}
		\begin{tikzpicture}[node/.style={circle,draw=black!100,fill=white!20,minimum size=15pt,inner sep=0pt},nonode/.style={circle}]
		\node[node] (d3) {$3$};
		\node[nonode] (d0) [above= of d3] {};
		\node[node] (d1) [left=0.5cm of d0] {$1$};
		\node[node] (d2) [right=0.5cm of d0] {$2$};
		
		  \draw[-latex] (d1) edge[bend left=10] (d2);
		    \draw[-latex] (d2) edge[bend left=10] (d1);
		    \draw[-latex] (d1) edge[bend left=10] (d3);
		    \draw[-latex] (d3) edge[bend left=10] (d1);
		    \draw[-latex] (d2) edge[bend left=10] (d3);
		    \draw[-latex] (d3) edge[bend left=10] (d2);
		
		\end{tikzpicture}
		\vspace{0.2cm}
    \caption{Lopsided association~\cite{harvey2015algorithmic}, and potential causality graph~\cite{achlioptas2014random}, when interpreted in the variable framework in the natural way.}
    \label{fig:hvai}
\end{center}\end{figure}

Now, it is not difficult to see that the \eqref{eq:asymmetric} condition applied to the graph corresponding to the dependency of Definition \ref{def:lops} requires for $\x_1,\x_2,\x_3\in(0,1)$ such that:
\begin{align*}
Pr[E_1] \leq & \x_1(1-\x_2)(1-\x_3), \\
Pr[E_2] \leq & \x_2(1-\x_1)(1-\x_3), \\
Pr[E_3] \leq & \x_3(1-\x_1)(1-\x_2). \\
\end{align*}
Thus, for the Moser and Tardos  lopsided  LLL to apply, it must hold that:$$2p-p^2\leq \x(1-\x)^2,$$ where $\x=\min\{\x_1,\x_2,\x_3\}$. This is maximized for $$\x=\frac{2^2}{3^3}=\frac{4}{27},$$ thus $p$ must be \emph{at most} $0.077$ (recall that our notion only requires for $p$ to be strictly less than one).

Finally, taking the dependency graph of $E_1$, $E_2$ and $E_3$, where two vertices are connected if their corresponding events' scopes intersect, we get as dependency graph  the cycle $C_3$ and henceforth by simple calculations, classical  Shearer's lemma requires that: $$1-3(2p-p^2)>0\Leftrightarrow p<0.184,$$ a stronger requirement  than the one that suffices to show that the undesirable events can be avoided through  our $d$-dependency notion. However, our version of Shearer's lemma (see Section \ref{sec:shearer}) gives  $p<1$.

Let us note that if we consider the classical definition of a lopsidependency graph  by Erd\H{o}s and Spencer \cite{erdosspencer1991}, namely a directed one satisfying inequality~\eqref{eq:lopsidependency}, then the graph turns out to be empty as well, since no event has a negative effect on any other, neither the union of any two does on the third. \hfill$\diamond$
\end{example}

We now show that the $d$-dependency graph is an Erd\H{o}s-Spencer lopsidependency graph.
\begin{lemma} \label{lem:genlops}
For any event $E_j$, $j=1,...,m$, let $I$ be a set of indices of events not in $\Gamma_j\cup \{j\}$. Then, it holds that:
$$\Pr\Big[E_j \ | \ \bigcap_{i \in I} \overline{E_i}\Big]\leq \Pr[E_j].$$ 
\end{lemma}

\begin{proof}
Let $E=\bigcap_{i \in I} \overline{E_i}$. Note that $E$ is not necessarily some of the $E_j$s (in fact if it is, then there is no assignment that avoids all the undesirable events). Now, in order to obtain a contradiction, suppose that:$$\Pr[E_j\mid E]>\Pr[E_j]$$ or, equivalently, that:$$\Pr[E_j\cap E]>\Pr[E_j]\cdot\Pr[E].$$ Then, it holds that:\begin{multline}\Pr[\E_j\cap E]=\Pr[E]-\Pr[E_j\cap E] <  \Pr[E]-\Pr[E_j]\cdot\Pr[E]=\\ \Pr[E](1-\Pr[E_j])=  \Pr[\E_j]\cdot\Pr[E].\label{ineq}
\end{multline}

Since $i \notin \Gamma_j$, for all $i\in I$, it holds that for any assignment $\alpha$ that makes $E_j$ and $E$ hold, there is no assignment $\beta$ that differs from $\alpha$ only in sc($E_j$) that makes $\E$ hold.

To obtain a contradiction, it suffices to show that:
\begin{equation*}
\Pr[E_j \ | \ E] \leq \Pr[E_j] \Leftrightarrow
\Pr[E_j \cap E] \leq \Pr[E_j]\cdot \Pr[E].
\end{equation*}
Suppose now $\alpha=(a_1,...,a_l),\beta=(b_1,...,b_l)$ are two assignments obtained by independently sampling the random variables twice, once to get $\alpha$ and once to get $\beta$.  It holds that:
$$\Pr[\underbrace{(\alpha, \beta): \text{ under } \alpha, \ E_j\cap E \text{ occurs and under }\beta, \ \E_j\text{occurs}}_{\text{event } S}]=\Pr[E_j\cap E]\cdot\Pr[\E_j].$$

Let now $\alpha'=(a'_1,...,a'_l)$, $\beta'=(b'_1,...,b'_l)$ be two assignments obtained by $\alpha$, $\beta$ by swapping values in variables in $\text{sc}(E_j)$:
\begin{itemize}
\item $a'_i=b_i$, for all $i$ such that $X_i \in \text{sc}(E_j)$, $a'_i=a_i$ for the rest,
\item $b'_i=a_i$, for all $i$ such that $X_i \in \text{sc}(E_j)$ and $b'_i=b_i$ for the rest.
\end{itemize}

Obviously $\alpha',\beta'$ are two independent samplings of all variables, since all individual variables were originally sampled independently, and we only changed the positioning of the individual variables. Also, under $\alpha'$, $\E_j$ occurs. Since none of the $E_i$'s is $d$-dependent on $E_j$, $E$ occurs under $\alpha'$. Also, under $b'$, $E_j$ occurs. Thus, it holds that:
\begin{multline*}\Pr[\underbrace{\text{under }\alpha', \ \E_j\cap E\text{ occurs and under } \beta', \ E_j \text{ occurs}}_{\text{event } T}]=\Pr[\E_j\cap E]\cdot\Pr[E_j]< \\ \Pr[\E_j]\cdot\Pr[E]\cdot\Pr[E_j],\end{multline*} where the last inequality holds by \eqref{ineq}.
Now, by the hypothesis and the construction of $\alpha'$, $\beta'$, it also holds that $S$ implies $T$. Thus:
\begin{align*}\Pr[S]\leq \Pr[T] \Leftrightarrow & \Pr[E_j \cap E]\cdot\Pr[\E_j]\leq \Pr[\E_j]\cdot\Pr[E]\cdot\Pr[E_j]\\
\Leftrightarrow & \Pr[E_j\cap E]\leq\Pr[E_j]\cdot\Pr[E].\end{align*}
The last inequality provides the contradiction and the proof is complete.
\qed\end{proof}

As is the case with all (lopsi-)dependency graphs defined based on notions in the variable framework, the dependency graph defined based on mutual independence of the events can be sparser than the $d$-dependency graph. Consequently, the same holds with the Erd\H{o}s-Spencer lopsidependency graph too. We end this section with a simple example that attests to that.
\begin{example}\label{sparser}
Let $X_1$ and $X_2$ be two independent random variables taking values, uniformly at random, in $\{0,1\}$.
Let also $E_1=\{X_1\neq X_2\}$ and $E_2=\{X_2=0\}$.

First, observe that $E_1$ is $d$-dependent on $E_2$. Indeed, let $\alpha=(0,0)$. Under $\alpha$, $E_1$ does not occur and $E_2$ does. By changing the value of $X_2\in\sco(E_2)$, we obtain the assignment $\beta=(0,1)$, under which $E_1$ occurs and $E_2$ doesn't. That $E_2$ is $d$-dependent on $E_1$ follows by taking the assignment $\beta$ and changing the value of $X_2\in\sco(E_1)$ to obtain $\alpha$.

On the other hand, notice that:$$\Pr[E_1]=\frac{1}{2}=\Pr[E_1\mid\E_2]$$ and that:$$\Pr[E_2]=\frac{1}{2}=\Pr[E_2\mid\E_1].$$ Thus $E_1$ and $E_2$ are independent.\hfill$\diamond$
\end{example}

\section{The lopsidependent case}\label{sec:lopsid}
Both approaches by   Moser \cite{moser2009} and by  Moser and Tardos \cite{mosertardos2010} search for an assignment that avoids the undesirable events by consecutively resampling the variables in the scopes of currently occurring events. In the approach of Moser \cite{moser2009}, when choosing the next event whose variables will be resampled, priority is given to the occurring events that belong to the \emph{extended neighborhood} in the dependency graph of the last resampled event (the extended neighborhood of an event $E$ is by definition the set of events sharing a variable  with $E$, with the event $E$ itself included). Thus the failure of the algorithm to return a correct answer within $n$ steps is depicted by a structure, called the witness forest of the algorithm's execution (will be formally defined below). So, in some sense, this approach guarantees that failure to produce results, will create, step after random step,  a structure out of randomness, something that cannot last for long, lest the second principle of thermodynamics is violated. This is, very roughly,  the intuition behind the entropic method. However, as we stressed, we analyze the algorithm by direct computations instead of referring to entropy. One key idea throughout this work is to give absolute priority, when searching  for the next event to be resampled, to the event itself, if it still occurs, in order to be able to utilize the $d$-dependency graph of the events.\remove{One key  idea throughout this work is to give more ``structure" to the witness forest by giving absolute priority, when searching  for the next event to be resampled, to the event itself, if it still occurs. Interestingly, this simple idea allows to consider a sparser dependency graph.} 

To be specific, see the pseudocode of \MAlg, which  successively produces random assignments, by resampling the variables in the scopes of occurring events, until it finds one under which no undesirable event occurs. When the variables in the scope of an occurring event $E_j$ are resampled, the algorithm checks if $E_j$ still occurs (lines \ref{m:if} and \ref{m:recursivecall1} of the \textsc{Resample} routine) and, only  in case it does not, looks for occurring events in $E_j$'s   neighborhood.  Finally, if and when all events in $E_j$'s neighborhood cease  occurring, the algorithm  looks for  still occurring events elsewhere.

\begin{algorithm}[h]\label{al:Malg}
\caption{$\textsc{M-Algorithm.}$}\label{alg:M-Alg}
\vspace{0.1cm}
\begin{algorithmic}[1]
\State Sample the variables $X_i$, $i=1,...,l$ and let $\alpha$ be the resulting assignment.\label{m:sample}
\While{there exists an event that occurs under the current assignment, let $E_j$ be the least indexed such event and}\label{m:while}
  \State{$\textsc{Resample}(E_j)$}\label{m:rootcall}
\EndWhile
\State  Output current assignment $\alpha$.\label{m:output}
\end{algorithmic}
\begin{algorithmic}[1]
\vspace{0.1cm}
\Statex \underline{$\textsc{Resample}(E_j)$}
\vspace{0.1cm}
\State Resample the variables in sc($E_j$).\label{m:resample}
\If{$E_j$ occurs}\label{m:if}
\State $\textsc{Resample}(E_j)$\label{m:recursivecall1}
\Else
\While{some event whose index is in $\G_j$ occurs under the current assignment, \Statex \hspace{1.35cm} let $E_k$ be the least indexed such event and}\label{m:recwhile}
  \State \textsc{Resample}($E_k$)\label{m:recursivecall2}
\EndWhile
\EndIf
\end{algorithmic}
\end{algorithm}
Obviously, if and when \MAlg \ stops, it produces an assignment to the variables for which none of the events occurs. Our aim now is to bound the probability that this algorithm lasts for at least $n$ steps. We count as a step an execution of  the variable resampling command {\sc Resample} in line \ref{m:resample} of the subroutine {\sc Resample}. 

Everywhere below the asymptotics are with respect to $n$, the number of steps, whereas the number $l$ of variables and the number $m$ of events are taken to be constants.

We first give some terminology, and then we start with a lemma that essentially guarantees that \MAlg \ makes progress. 

A \emph{round} is the duration of any \textsc{Resample} call during an execution of \MAlg. Rounds are nested. The number of nested rounds completed coincides with the number  of steps the algorithm takes.  A \textsc{Resample} call made from line \ref{m:rootcall} of the main algorithm is a \emph{root call}, while one made from within another call is a \emph{recursive call}. 

\begin{lemma} \label{lem:mprogr}
Consider an arbitrary call of \textsc{Resample}($E_j$). Let $\X_j$ be the set of events that  do not occur at the start of this call. Then, if and when this call terminates, all events in $\X_j\cup\{E_j\}$ do not occur. \end{lemma}
\begin{proof} 
Without loss of generality, say that \textsc{Resample}($E_j$) is the root call of {\sc Resample}, suppose it terminates and that $\alpha$ is the produced assignment of values. Furthermore, suppose that $E_k\in \X_j\cup\{E_j\}$ and that $E_k$ occurs under $\alpha$.

Let $E_k\in \mathcal{X}_j$. Then, under the assignment at the beginning of the main call, $E_k$ did not occur. Thus, it must be the case that at some point during this call, a resampling of some variables caused $E_k$ to occur. Let $\textsc{Resample}(E_s)$ be the last time $E_k$ became occurring, and thus remained occurring until the end of the main call.

Since $E_k$ did not occur at the beginning of $\textsc{Resample}(E_s)$, there is an assignment of values $\alpha$ such that $E_s, \E_k$ occur. Furthermore, for the main call to have terminated, $\textsc{Resample}(E_s)$ must have terminated too. For this to happen $\textsc{Resample}(E_s)$ must have exited lines \ref{m:if} and \ref{m:recursivecall1} of its execution. During this time, only variables in sc($E_s$) were resampled and at the end, $E_s$ did not occur anymore. Thus, $E_k$ is in the neighborhood of $E_s$. But then, by line \ref{m:recwhile} of the \textsc{Resample} routine, $\textsc{Resample}(E_s)$ couldn't have terminated and thus, neither could the main call. Contradiction.

Thus $E_k=E_j$. Since under the assignment at the beginning of the main call, $E_j$ occurred, by lines \ref{m:if} and \ref{m:recursivecall1} of the \textsc{Resample} routine, it must be the case that during some resampling of the variables in sc($E_j$), $E_j$ became non-occurring. The main call could not have ended after this resampling, since $E_j$ occurs under the assignment $\beta$ produced at the end of this call. Then, there exists some $r\in \G_j$ such that \textsc{Resample}($E_r$) is the subsequent \textsc{Resample} call. Thus $E_j\in \X_r$ and we obtain a contradiction as in the case where $E_k\in \X_j$ above.\end{proof}
An immediate corollary of Lemma \ref{lem:mprogr}, is that the events of the root calls of \textsc{Resample} are pairwise distinct, therefore there can be \emph{at most} $m$ such root calls in any execution of \MAlg.

Consider now \emph{rooted forests}, i.e. forests of trees such that each tree has a special node designated as its root, whose vertices are labeled by events $E_j$, $j\in\{1,\ldots,m\}$. We will use such forests to depict the executions of \MAlg.
\begin{definition}\label{def:feasibleforest}
A labeled rooted forest $\F$ is called \emph{feasible} if: \begin{enumerate}
\item the labels of its roots are \emph{pairwise distinct}, \item the labels of any two siblings (i.e. vertices with a common parent) are distinct and \item an internal vertex labeled by $E_j$ has at most $|\G_j|+1$ children, with labels whose indices are in $\G_j\cup\{j\}$. \remove{\item an internal vertex labeled by $E_j$ has either one child labeled again by $E_j$ or at most $|\G_j|$ children, with labels whose indices are in $\G_j$. We call the former {\em idle nodes} and the latter {\em advance} nodes. Non-internal nodes (leaves) count as either idle or advance, but not both.}
\end{enumerate}\end{definition}
The number of nodes of a feasible forest $\mathcal{F}$ is denoted by $|\mathcal{F}|$.

The nodes of such a labeled forest are ordered as follows: children of the same node are ordered as their labels are; nodes in the same tree are ordered by preorder (the ordering induced by running the \emph{depth-first search} algorithm on input such a tree), respecting the ordering between siblings; finally if the label on the root of a tree $T_1$ precedes the label of the root of $T_2$, all nodes of $T_1$ precede all nodes of $T_2$.

Given an execution of \MAlg \ that lasts for \emph{at least} $n$ rounds, we construct, in a unique way, a feasible forest with $n$ nodes. First, we create one node for each \textsc{Resample} call and label it with its argument. Root calls correspond to the roots of the trees and a recursive call made from line \ref{m:recursivecall1} or \ref{m:recursivecall2} of a \textsc{Resample}($E_j$) call gives rise to a child of the corresponding node of this \textsc{Resample}($E_j$) call. We say that a feasible forest $\F$ constructed this way is the $n$-\emph{witness} forest of \MAlg's execution and we define $W_{\F}$ to be the event \MAlg \  executes producing $\F$ as an $n$-witness forest. \remove{Note that Definition \ref{def:dd} is essential to have that the children of a node labeled with an event are either only one node with the same label, or at most the number of elements in the outwards neighborhood of $E$ in the $d$-dependency graph, labeled with the corresponding neighboring events.} 

Define $P_n$ to be the probability that \MAlg \ lasts for \emph{at least} $n$ rounds. Obviously:
\begin{equation}\label{Pn}
   P_n =\Pr\Bigg[\bigcup_{\mathcal{F}: |\mathcal{F}|=n}W_\mathcal{F}\Bigg] = \sum_{\mathcal{F}: |\mathcal{F}|=n}\Pr\Big[ W_\mathcal{F}\Big],\end{equation}
where the last equality holds because the events $W_{\F}$ are disjoint.

Unfortunately,  \MAlg \ introduces various dependencies that render the probabilistic calculations essentially impossible.   For example, suppose that the $i$-th node of a witness forest $\F$ is labeled by $E_j$ and its children have labels with indices in $\G_j$. Then, under the assignment produced at the end of the $i$-th round of this execution, $E_j$ does not occur.

To avoid such dependencies, we introduce a \emph{validation algorithm}, \textsc{ValAlg}. Interestingly, \textsc{ValAlg} produces no progress towards locating the sought after assignment. However, as we will see, it has two useful properties: (i) From round to round, the distribution of the variables does not change, a fact that makes possible a direct  probabilistic analysis  and (ii) the probability that it lasts for at least $n$ steps bounds from above the respective probability of \MAlg \ (see Lemma \ref{lem:hatPn}).

\begin{algorithm}[H]
\caption{$\textsc{ValAlg.}$}\label{alg:valalg}
\vspace{0.1cm}
\begin{algorithmic}[1]
\Statex \underline{\textbf{Input:}} Feasible forest $\F$ with labels $E_{j_1},\ldots,E_{j_n}$.
\vspace{0.1cm}
\State Sample the variables $X_i$, $i=1,...,l$.\label{valalg:sample}
\For{s=1,\ldots,n}\label{valalg:for}
  \If{$E_{j_s}$ does not occur under the current assignment} \label{valalg:if}
     \State \textbf{return} {\tt failure} and exit. \label{valalg:fail}
  \Else
      \State Resample the variables in sc($E_{j_s}$)\label{valalg:resample}
  \EndIf
\EndFor \label{valalg:endfor}
\State \textbf{return} {\tt success}.\label{valalg:success}
\end{algorithmic}
\end{algorithm}

A round of \Val \ is the duration of any \textbf{for} loop executed at lines \ref{valalg:for}-\ref{valalg:endfor}. If the algorithm manages to go through its input without coming upon a non-occurring event at any given round, it returns {\tt success}. Thus, the success of  \Val \ has no consequence with respect to the occurrence, at the end,  of the undesirable events.  

The following result concerns the distribution of the random assignments at any round of \Val.
\begin{lemma}[Randomness lemma] \label{lem:randomness}
At the beginning of any given round of \Val, the distribution of the current assignment of values to the variables $X_i$, $i=1,...,l$, given that \Val \ has not failed, is as if all variables have been sampled anew.
\end{lemma} 
\begin{proof}
This follows from the fact that at each round, the variables for which their values have been exposed, are immediately resampled.
\end{proof}
Now, given a feasible forest $\F$ with $n$ nodes, we say that $\F$ is \emph{validated} by \Val \ if the latter returns {\tt success} on input $\F$. The event of this happening is denoted by $V_{\F}$. We also set:
\begin{equation}\label{eq:sumhatPn}
\hat{P}_n= \sum_{\F:\ |\F|=n} \Pr[V_\mathcal{F}].
\end{equation}
\begin{lemma} \label{lem:hatPn}
For any feasible forest $\F$, the event $W_{\F}$ implies the event $V_{\F}$, therefore $P_n \leq \hat{P}_n.$
\end{lemma}
\begin{proof}
Indeed, if the random choices made by an execution of \MAlg \ that produces as witness forest $\F$ are made by \Val \ on input $\F$ , then clearly \Val \ will return {\tt success}.
\end{proof}

From now on, we will use the following notation:  $\bn=\{n_1,\ldots,m\}$, where $n_1,\ldots,n_{m}\geq 0$ are such that $\sum_{i=1}^{m} n_i=n$ and $\bn-(1)_j:=(n_1,\ldots,n_j-1,\ldots,n_{m})$\remove{$n_1,\ldots,n_{2m}\geq 0$ are such that $\sum_{i=1}^{2m} n_i=n$ and $\bn-(1)_j:=(n_1,\ldots,n_j-1,\ldots,n_{2m})$}. We now state and prove our first result:

\begin{customthm}{\ref{thm:dalll}a}[Algorithmic directed  LLL]\label{thm:algalll}
Suppose that there exist $\x_1,\x_2,\ldots,\x_m\in(0,1)$, such that $$\Pr(E_j)\leq \x_j \prod_{i\in \G_j} (1-\x_i),$$
for all $j\in\{1,\ldots,m\}$, where $\G_j$ denotes the outwards neighborhood of $E_j$ in the $d$-dependency graph. Then, the probability that \MAlg \  %\textsc{M-Algorithm} \ref{alg:M-Alg} 
executes for \emph{at least} $n$ rounds is \emph{inverse exponential} in $n$.\remove{, ignoring polynomial in $n$ factors, \emph{at most}: \begin{equation}\label{eq:dallleq}(1-\x)^{\sum_{i=1}^{2m} n_i}=(1-\x)^n, \mbox{ with } \x=\min_{i=1,\ldots,m}\{\x_i\} <1.\end{equation}}
\end{customthm}
\begin{proof}
We may assume, without loss of generality,  that $\Pr[E_j] <  \x_j\prod_{i\in\G_j}(1-\x_i)$ for all $j\in\{1,\ldots,m\}$, i.e. that the hypothesis is given in terms of a strict inequality. Indeed, otherwise consider an event $B$, such that $B$ and $E_1,\ldots,E_m$, are mutually independent, where $\Pr[B]=1-\delta$, for arbitrary small $\delta>0$. We can now perturb the events a little, by considering e.g. $E_j\cap B$, $j=1,\ldots,m$. As a consequence, we can also assume without loss of generality that for some  other small enough $\epsilon>0$, we have that $Pr[E_j] \leq (1-\epsilon)\x_j\prod_{i\in\G_j}(1-\x_i)$.

By Lemma \ref{lem:hatPn}, it suffices to prove that $\hat{P}_n$ is inverse exponential in $n$. Specifically, we show that $\hat{P}_n\leq (1-\epsilon)^n$. Let  $Q_{\bn,j}$ be the probability that \Val \, is successful when started  on a {\em tree} whose root is labeled with $E_j$ and has   $\sum_{i=1}^{m} n_i=n$ nodes labeled with $E_1, \ldots, E_m$. Observe that to obtain a bound for $\hat{P}_n$ we need to add over all possible forests with $n$ nodes in total. Thus, it holds that:
$$\hat{P}_n\leq\sum_{\bn}\sum_{\bn^1+\ldots+\bn^m=\bn}\Big(Q_{\bn^1,1}\cdots Q_{\bn^m,m}\Big).$$ Our aim is to show that $Q_{\bn,j}$ is exponentially small to $n$, for any given sequence of $\bn$ and any $j\in\{1,\ldots,m\}$. Thus, by ignoring polynomial in $n$ factors, the same will hold for $\hat{P}_n$ (recall that the number of variables and the number of events are considered constants, asymptotics are in terms of the number of steps $n$ only).
 
Let $\G_j^+:=\G_j\cup\{j\}$, and assume that, for each $j\in\m$, $|\G_j^+|=k_j$. Observe now that $Q_{\bn,j}$ is bounded from above  by a function, denoted again by $Q_{\bn,j}$ (to avoid overloading the notation), which follows the recurrence:
\begin{equation} \label{multqnj}
Q_{\bn,j} =   \Pr[E_j]\cdot \sum_{\bn^1+\cdots+\bn^{k_j}=\bn-(1)_j}  \Big(Q_{\bn^1,j_1}+\cdots Q_{\bn^{k_j},j_{k_j}}\Big),
\end{equation}
with initial conditions  $Q_{\bn,j}=0$ when $n_j=0$ and there exists an $i\neq j$ such that $n_i\geq 1$; and with  $Q_{\mathbf{0},j}=1$, where $\mathbf{0}$ is a sequence of $m$ zeroes.

To solve the above recurrence,  we introduce, for $j=1,\ldots,m$, the 
 \emph{multivariate generating functions}: \begin{equation} \label{mult:OGF} Q_j(\bt)=\sum_{\bn:n_j\geq 1} Q_{\bn,j}\bt^{\bn},\end{equation}  
where $\bt=(t_1,\ldots,t_{m})$, $\bt^{\bn}:=t_1^{n_1}\cdots t_{m}^{n_{m}}$.

By multiplying both sides of \eqref{multqnj} by $\bt^n$ and adding all over suitable $\bn$, we get the system of equations $\Q$:
\begin{equation}\label{system}
Q_j(\bt)= t_j f_j(\Q),
\end{equation}
where, for $\bx=(x_1,\ldots,x_{m})$ and $j=1\ldots,m$:
\begin{equation}\label{fQj}
f_j(\bx)= (1-\e)\cdot\x_j\cdot\Bigg(\prod_{i\in \G_j}(1-\x_i)\Bigg)\cdot\Bigg(\prod_{i\in\G_j^+}(x_i+1)\Bigg).
\end{equation}
To solve the system, we will directly use the result of Bender and Richmond in \cite{bender1998multivariate} (Theorem $2$). Let $g$ be any $m$-ary projection function on some of the $m$ coordinates. In the sequel we take $g := pr^{m}_s$, the $(m)$-ary projection on the $s$-th coordinate. Let also $\Btree$ be the set of trees $B=(V(B),E(B))$ whose vertex set is $\{0,1,\ldots,m\}$ and with edges directed towards $0$. By \cite{bender1998multivariate}, we get: \begin{equation}\label{BR}
[{\bt}^{\bn}] g((\Q,\R)(\bt))= \\ \frac{1}{\prod_{j=1}^{m} n_j} \sum_{B \in \Btree}[\bx^{\bn-\1}] \frac{\partial(g,f_1^{n_1},\ldots, f_{m}^{n_{m}})}{\partial B},\end{equation}
where the term for a tree $B\in\Btree$ is defined as:
\begin{equation}\label{bendertree}
[\bx^{\bn-\1}]\prod_{r\in V(B)}\Bigg\{\Bigg(\prod_{(i,r)\in E(B)}\frac{\partial}{\partial x_i}\Bigg)f_r^{n_r}(\bx)\Bigg\},
\end{equation}
where $r\in\{0,\ldots,m\}$ and $f_0^{n_0}:=g$.

We consider a tree $B\in\Btree$ such that \eqref{bendertree} is not equal to $0$. Thus, $(i,0)\neq E(B)$, for all $i\neq s$. On the other hand, $(s,0)\in E(B)$, lest vertex $0$ is isolated, and each vertex has out-degree \emph{exactly} one, lest a cycle is formed or connectivity is broken. From vertex $0$, we get $\frac{\partial pr_s^{m}(\bx)}{\partial x_s}=1$. Since our aim is to prove that $\hat{P}_n$ is exponentially small in $n$, we are are interested only in factors of \eqref{bendertree} that are exponential in $n$, and we can thus ignore the derivatives (except the one for vertex $0$), as they introduce only polynomial (in $n$) factors to the product. Thus, we have that \eqref{bendertree} is equal to the coefficient of $\bx^{\bn-\1}$ in:
\begin{equation}\label{tree}
  \prod_{j=1}^m\Bigg\{(1-\e)^{n_j}\cdot\x_j^{n_j}\cdot\Bigg(\prod_{i\in \G_j}(1-\x_i)^{n_j}\Bigg)\cdot\Bigg(\prod_{i\in\G_j^+}(x_i+1)^{n_j}\Bigg)\Bigg\}.
\end{equation}
We now group the factors of \eqref{tree} according to the $i$'s. We have already argued each vertex $i$ has out-degree $1$. Thus, the exponent of the term $x_i+1$ is $n_i+\sum_{j:i\in \G_j} n_j$ and the product of \eqref{tree} is equal to:
\begin{equation}\label{group}
  \prod_{i=1}^m\Bigg\{(1-\e)^{n_i}\cdot\x_i^{n_i}\cdot(1-\x_i)^{\sum_{j:i\in\G_j}n_j}\cdot(x_i+1)^{n_i+\sum_{j:i\in\G_j}n_j}\Bigg\}.
\end{equation}
Using the binomial theorem and by ignoring polynomial factors, we get that the coefficient of $\bx^{\bn-\1}$ in \eqref{group} is:
\begin{equation}\label{binomial}
  \prod_{i=1}^m\Bigg\{(1-\e)^{n_i}\cdot\x_i^{n_i}\cdot(1-\x_i)^{\sum_{j:i\in\G_j}n_j}\cdot\binom{n_i+\sum_{j:i\in\G_j}n_j}{n_i}\Bigg\}.
\end{equation}
By expanding $(\x_i+1-\x_i)^{n_i+\sum_{j:i\in \G_j}n_j}$, we get that \eqref{binomial} is at most:
\begin{equation}\label{final}
\prod_{i=1}^m (1-\e)^{n_i}=(1-\e)^{\sum_{i=1}^n n_i}=(1-\e)^n.
\end{equation}
Thus, $\hat{P}_n$ is inverse exponential in $n$.
\end{proof}

From Theorem \ref{thm:algalll}, the existential Theorem \ref{thm:dalll} immediately follows.

\section{Shearer's lemma}\label{sec:shearer}

We now turn our attention to Shearer's lemma. The first algorithmic proof for general probability spaces was given by Kolipaka and Szegedy \cite{kolipakaszegedy2011}. Harvey and Vondr\'ak \cite{harvey2015algorithmic} proved a version of the lemma for their lopsided association graphs (again in the generalized framework).

Here, we apply it to the underlying \emph{undirected} graph of the $d$-dependency graph we introduced in Section \ref{sec:d-depend}. Our work is situated in the variable framework and we give a forward argument that directly leads to an exponentially small bound of the probability of the algorithm lasting for at least $n$ steps.

Let $E_1,\ldots,E_m$ be events, whose vector of probabilities is $\p=(p_1,\ldots,p_m)$, that is $Pr[E_j]=p_j\in(0,1), \ j=1,\ldots,m$. Let also $G=\langle \{1,\ldots,m\},E\rangle$ be a graph on $m$ vertices, where we associate each event $E_j$ with vertex $j$, $j=1,\ldots,m$ and where $$E=\{\{i,j\}\mid\text{either }E_j\text{ is }d\text{-dependent on }E_i\text{ or }E_i\text{ is }d\text{-dependent on }E_j\}.$$ For each vertex $j$, we denote its \emph{neighborhood} in $G$ by $\G_j$, $j=1,\ldots,m$.

A subset $I\subseteq \{1,\ldots,m\}$ of the graph's vertices is an \emph{independent set} if there are no edges between its vertices. Abusing the notation, we will sometimes say that an independent set $I$ contains events (instead of indices of events). Let $I(G)$ denote the set of independent sets of $G$. For any $I\in I(G)$, let $\G(I):=\bigcup_{j\in I}\G_j$ be the set of neighbors of the vertices of $I$. Following \cite{kolipakaszegedy2011}, we say that $I$ \emph{covers} $J$ if $J\subseteq I\cup\G(I)$.

A \emph{multiset} is usually represented as a couple $(A,f)$, where $A$ is a set, called the \emph{underlying set}, and $f:A\mapsto\N_{\geq 1}$ is a function, with $f(x)$ denoting the multiplicity of $x$, for all $x\in A$. In our case, the underlying sets of multisets are always subsets of $\{1,\ldots,m\}$. Thus, to make notation easier to follow, we use couples $(I,\z)$, where $I\subseteq\{1,\ldots,m\}$ and $\z=(z_1,\ldots,z_m)$ is an $m$-ary vector, with $z_j\in\N$ denoting the multiplicity of $E_j$ in $I$. Note that $z_j=0$ if and only if $j\notin I$, $j=1,\ldots,m$.

Consider our second main theorem below, which is a variation of Shearer's Lemma for $d$-dependency graphs.

\begin{theorem}[Shearer's lemma for $d$-dependency graphs]\label{thm:shearer} If for all $I\in I(G)$: \begin{equation}\label{lopshearer}\tag{\sf{Shear}}q_I(G,\bar{p})=\sum_{J\in I(G): \ I\subseteq J} (-1)^{|J\setminus I|}\prod_{j\in J} p_j> 0,\end{equation}
then $$Pr\Bigg[\bigwedge_{j=1}^m\bar{E}_j\Bigg]>0.$$
\end{theorem}
Actually, we prove below an algorithmic version  (Theorem \ref{thm:algshearer}) of the existential Theorem~\ref{thm:shearer}, where we  give exponentially small estimates of the probability of the algorithm not producing the desired results.

The algorithm we use is a variation of the \textsc{Maximal Set Resample} algorithm, designed by Harvey and Vondr\'ak in \cite{harvey2015algorithmic}, which is a slowed down version of the algorithm in \cite{kolipakaszegedy2011}. The algorithm constructs multisets whose underlying sets are independent sets of $G$, by selecting occurring events that it resamples until they do not occur anymore.

\begin{algorithm}[h]
\caption{$\textsc{MaxSetRes.}$}\label{alg:maxsetres}
\vspace{0.1cm}
\begin{algorithmic}[1]
\State Sample the variables $X_i$, $i=1,...,l$ and let $\alpha$ be the resulting assignment.\label{msr:sample}
\State $t:=1$, $I_t:=\emptyset$, $\z^t:=(0,\ldots,0)$. \label{msr:initial}
\Repeat \label{msr:repeat}
\While{there exists an event $E_j\notin I_t\cup\Gamma(I_t)$ that occurs under the current assignment, \Statex \hspace{0.8cm} let $E_j$ be the least indexed such event and}\label{msr:main-while}
 \State $I_t:=I_t\cup\{j\}$, $c:=0$, $z^t_j:=1$.
 \While  {$E_j$  occurs }\label{msr:inner-while}
  \State Resample the variables in sc($E_j$).\label{msr:idle}
  \State $c:=c+1$.
 \EndWhile . \label{msr:inner-endwhile}
 \If{there exists an occurring event that is not in $I_t\cup\Gamma(I_t)$}\label{msr:if}
 \State $z_j^t:=c+1$.
 \ElsIf{there exists an occurring event not in $\Gamma_j$}\label{msr:elseif}
 \State $t:=t+1$, $I_t:=\emptyset$.\label{msr:eventind}
 \State $z_j^t:=c$.
 \Else
 \For{$s=1,\ldots,c$}\label{msr:idlephases} 
 \State $I_{t+s}:=\{E_j\}$.
 \EndFor
 \State $t:=t+c+1$, $I_t:=\emptyset$.\label{msr:re-initialize}
\EndIf \label{msr:endif}
 \EndWhile
 \Until $I_t=\emptyset$. \label{msr:until}
\State  Output current assignment $\alpha$.\label{msr:output}
\end{algorithmic}
\end{algorithm}

A \emph{step} of \MSR \ is a single resampling of the variables of an event in line~\ref{msr:idle}, whereas a \emph{phase} is an iteration of \textbf{repeat} at lines \ref{msr:repeat}--\ref{msr:until}, except from the last iteration that starts and ends with $I_t= \emptyset$. Phases are not nested. During each phase, there are at most $m$ repetitions of the \textbf{while}--loop of lines \ref{msr:inner-while}--\ref{msr:inner-endwhile}, where $m$ is the number of events (recall that the number of variables $l$, and the number of events $m$ are considered to be constants). During each phase, a multiset $(I_t,z^t)$ is created, where the underlying set $I_t$ is an independent set of $G$. There are two scenarios that can happen at the end of a phase. The first is by line \ref{msr:eventind}, where \MSR \ creates a new independent set, containing $c$ copies of the last event it resampled at lines~\ref{msr:inner-while}--\ref{msr:inner-endwhile}. The second is by lines~\ref{msr:idlephases}--\ref{msr:re-initialize}, where \MSR \ proceeds $c$ phases at once, creating $c$ singleton sets, containing only the event it lastly resampled at lines~\ref{msr:inner-while}--\ref{msr:inner-endwhile}. Lines~\ref{msr:if}--\ref{msr:endif} exist for technical reasons that will become apparent later. 

 Note that, by lines \ref{msr:main-while} and \ref{msr:until}, if and when \MSR \ terminates, it produces an assignment of values under which none of the events occurs.

We now proceed with the first lemma concerning the execution of \MSR. Note that it refers to the underlying independent sets of the multisets created at each phase.

\begin{lemma}\label{lem:msrstable}
$I_{t}$ \emph{covers} $I_{{t+1}}$, for all $t\in \{1,\ldots,n-1\}$.
\end{lemma}
\begin{proof}
Let $E_j$ be an event in $I_{{t+1}}$. Then, at some point during phase $t+1$, $E_j$ was occurring. We will prove below  that $E_j$ occurs also at the beginning of phase $t+1$. This will conclude the proof, since if  $E_j \not \in \G(I_t) \cup I_t$, then  at the moment when  phase $I_{t+1}$ was to start, the algorithm instead of starting $I_{t+1}$ would opt to add  $E_j$ to $I_t$, a contradiction.  

Assume that $I_{t+1}\neq\{E_j\}$, lest we have nothing to prove. To prove that $E_j$ occurs at the beginning of phase $t+1$, assume towards a contradiction that it does not. Then it must have become occurring during the repeated resamplings of an event $E_r$ introduced into $I_{t+1}$.

Therefore, under the assignment when $E_r$ was selected, $E_r$ occurred and $E_j$ did not. Furthermore, during the repeated resamplings of $E_r$, only variables in sc($E_r$) had their values changed, and under the assignment at the end of these resamplings, $E_r$ ceases occurring and $E_j$ occurs.  By Definition \ref{def:dd}, $E_j$ is $d$-dependent on $E_r$ and thus $E_j\in\Gamma(I_{t+1})$. By lines \ref{msr:main-while}, \ref{msr:if} and \ref{msr:elseif}, $E_j$ could not have been selected at any point during round $t+1$. 
This concludes the proof.
\end{proof}
Consider the following definition:

\begin{definition}[Kolipaka and Szegedy \cite{kolipakaszegedy2011}]
A {\em stable} sequence of events is a sequence of non-empty independent sets  $\mathcal{I} = I_{1},\ldots,I_{n}$ such that  $I_{t}$ covers $I_{{t+1}}$, for all $t\in\{1,\ldots,n-1\}$.\end{definition}
Stable sequences play the role of witness structures in the present framework. By Lemma \ref{lem:msrstable}, the underlying sets of the sequence of multisets \MSR \ produces in an execution that lasts for at least $n$ phases, is a stable sequence of length $n$.

We now prove:
\begin{customthm}{\ref{thm:shearer}a}[Algorithmic Shearer's lemma for $d$-dependency graph]\label{thm:algshearer} If for \newline all $I\in I(G)$: $$q_I(G,\bar{p})=\sum_{J\in I(G): \ I\subseteq J} (-1)^{|J\setminus I|}\prod_{j\in J} p_j> 0,$$  then the probability $\rP_n$ that \MSR \ lasts for at least $n$ phases is exponentially small, i.e.  for some constant $c<1$,  $\rP_n$  is at most $c^n$, ignoring polynomial factors. 
\end{customthm}

Again, Theorem \ref{thm:shearer} follows immediately from Theorem \ref{thm:algshearer}.
\begin{proof}
Let $\bz=(\z^1,\ldots,\z^n)$ be an $n$-ary vector, whose elements $\z^t=(z_1^t,\ldots,z_m^t)$ are $m$-ary vectors of non-negative integers, $t=1,\ldots,n$. Let also $$(\I,\bz)=(I_1,\z^1),\ldots,(I_n,\z^n)$$ be a sequence of multisets, whose underlying sequence $\I$ is a stable sequence. We denote by $|(\I,\bz)|$ its length, i.e. the number of pairs $(\I_t,\z^t)$ it contains. If $\rP(\I,\bz)$ is the probability that an execution of \MSR \ produced $(\I,\bz)$ (which can be zero), it is easy to see that:
\begin{equation}\label{multiset}\rP_n = \sum_{(\I,\bz): |(\I,\bz)| =n} \rP(\I,\bz),\end{equation} where the sum is over all possible pairs of stable sequences $\I$ of length $n$ and vectors $\bz$.

To bound the rhs of \eqref{multiset}, consider the validation algorithm \MSV \ below. \MSV, on input a stable sequence $\mathcal{I}=I_1,\ldots,I_n$, proceeds to check each event contained in each independent set. If this event does not occur, it fails; else it resamples the variables in its scope. Note that the success or failure of this algorithm has nothing to do with finding an assignment such that none of the events occur.

\begin{algorithm}[h]
\caption{$\textsc{MaxSetVal.}$}\label{alg:maxsetval}
\vspace{0.1cm}
\begin{algorithmic}[1]
\Statex \underline{\textbf{Input:}} Stable sequence $\mathcal{I} =I_{1},\ldots,I_{n}$.
\vspace{0.1cm}
\State Sample the variables $X_i$, $i=1,...,l$.
\For{t=1,\ldots,n}\label{msv:for}
   \For{each event $E_j$  of $I_{t}$ }
  \If{$E_j$ does not occur under the current assignment} 
     \State \textbf{return} {\tt failure} and exit.
  \Else
      \State Resample the variables in $\rm{sc}(E_j)$
  \EndIf
  \EndFor
\EndFor \label{msv:end}
\State \textbf{return} {\tt success}.
\end{algorithmic}
\end{algorithm}
A phase of \MSV \ is any repetition of lines \ref{msv:for}--\ref{msv:end}. Let $\hat{\rP}(\mathcal{I})$ be the probability that  \MSV \ is successful on input $\mathcal{I}$ and:  \begin{equation}\label{msvpn}\hat{\rP}_n := \sum_{\mathcal{I}:|\mathcal{I}| =n } \hat{\rP}(\mathcal{I}).\end{equation} 
To obtain our result, we now proceed to show: (i) that $\rP_n\leq\hat{\rP}_n$ and (ii) that $\hP_n$ is inverse exponential to $n$. For the former, consider the validation algorithm \MV, algorithm \ref{alg:multisetval}, below.

\begin{algorithm}[ht]
\caption{$\textsc{MultiSetVal.}$}\label{alg:multisetval}
\vspace{0.1cm}
\begin{algorithmic}[1]
\Statex \underline{\textbf{Input:}} $(\I,\bz)=(I_1,z^1),\ldots,(I_n,z^n)$, $I_t=\{E_{t_1},\ldots,E_{t_{k_t}}\}$, $t=1,\ldots,n$.
\vspace{0.1cm}
\State Sample the variables $X_i$, $i=1,...,l$.
\For{$t=1,\ldots,n$}\label{multi:for}
\For{$s=1,\ldots,k_t-1$}
\For{$r=1,\ldots,z^t_{t_s}$}
  \If{$E_{t_s}$ does not occur under the current assignment}
     \State \textbf{return} {\tt failure} and exit.
  \Else
      \State Resample the variables in sc($E_{t_s}$)
  \EndIf
  \EndFor
  \If{$E_{t_s}$ occurs under the current assignment}\label{multi:occur}
     \State \textbf{return} {\tt failure} and exit.
     \EndIf
  \EndFor
  \If{$E_{t_{k_t}}$ does not occur under the current assignment}
     \State \textbf{return} {\tt failure} and exit.
  \Else
      \State Resample the variables in sc($E_{t_{k_t}}$)
  \EndIf
\EndFor \label{multi:end}
\State \textbf{return} {\tt success}.
\end{algorithmic}
\end{algorithm}

\MV \ takes as input a sequence $(\I,\bz)$ of multisets whose underlying sequence is stable. It then proceeds, for each multiset, to check if its events occur under the current assignment it produces. If not it fails, else it proceeds. When the last copy of an event inside a multiset, apart from the last event, is resampled, it checks if that event still occurs (line \ref{multi:occur}). If it does, the algorithm fails. If it manages to go through the whole sequence without failing, it succeeds. Note again that the success or failure of \MV \ has nothing to do with obtaining an assignment such that none of the events holds. 

We call a \emph{phase} of \MV \ each repetition of lines~\ref{multi:for}--\ref{multi:end}. Let also $\tP(\I,\bz)$ be the probability that \MV \ succeeds on input $(\I,\bz)$. We prove two lemmas concerning \MV.
\begin{lemma}\label{lem:msr-multi}
For each sequence $(\I,\bz)$, $\rP(\I,\bz)\leq \tP(\I,\bz)$. Thus:\begin{equation}\label{msr-multi}
    \rP_n \leq \sum_{(\I,\bz): |(\I,\bz)| =n} \tP(\I,\bz)
\end{equation}
\end{lemma}
\begin{proof}
It suffices to prove the first inequality, as the result is then derived by \eqref{multiset}. Note that the last event in every multiset that \MSR \ produces always has multiplicity $1$ and that, furthermore, it is not required to be non-occurring after resamping it, in contrast with all the other events in the multiset. It is now straightforward to notice that if \MV \ makes the same random choices as \MSR \ did when it created any sequence $(\I,\bz)$, \MV \ will succeed on input $(\I,\bz)$.
\end{proof}
\begin{lemma}\label{lem:mv-msv}
For any $(\I,\bz)$, it holds that:\begin{equation}\label{mv-msv}
    \sum_{(\I,\bz): |(\I,\bz)| =n} \tP(\I,\bz)= \sum_{\I:|\I|=n}\hP(\I).
\end{equation}
\end{lemma}
\begin{proof}
We will rearrange the sum in the lhs of \eqref{mv-msv}. Assume that the stable sequences of length $n$ in $G$ are arbitrarily ordered as $\I_1,\ldots,\I_s$. Then, it holds that:\begin{equation}\label{terms}
\sum_{(\I,\bz):|(\I,\bz)|=n}\tP(\I,\bz)=\sum_{\bz=(\z^1,\ldots,\z^n)}\tP(\I_1,\bz)+\cdots+\sum_{\bz=(\z^1,\ldots,\z^n)}\tP(\I_s,\bz).    
\end{equation}
Let $\I=(I_1,\ldots,I_n)$ be a stable sequence and consider the term $$\sum_{\bz=(\z^1,\ldots,\z^n)}\tP(\I,\bz)$$ of \eqref{terms} corresponding to $\I$. It suffices to show that is is equal to $\hP(\I)$.

Assume again that $I_t=\{E_{t_1},\ldots,E_{t_{k_t}}\}$ and that $\z^t=(z^t_1,\ldots,z^t_m)$, where $z^t_j\geq 0$, $j=1,\ldots,m$, $t=1,\ldots,n$. Finally, set: $$\tP(I_t,{\z^t}):=\Pr[E_{t_1}]^{z^t_1}\Pr[\E_{t_1}\cap E_{t_2}]\Pr[E_2]^{z^t_2-1}\cdots\Pr[E_{t_{k_t}-1}]^{z^t_{k_t-1}}\Pr[\E_{t_{k_t}-1}\cap E_{t_{k_t}}].$$
Then, it holds that:
\begin{equation}\label{sum-prod}
    \sum_{\bz=(\z^1,\ldots,\z^n)}\tP(\I,\bz) =\sum_{\bz=(z^1,\ldots,z^n)}\prod_{t=1}^n\tP(I_t,z^t).
\end{equation}

By Lemma \ref{lem:genlops}, it holds that all the factors $\Pr[\E\cap E']$ that appear in \eqref{sum-prod} are less or equal than $\Pr[\E]\cdot\Pr[E']$. Now, by factoring out:$$\hP(\I)=\prod_{t=1}^n\Bigg(\Pr[E_{t_1}]\cdots\Pr[E_{t_{k_t}}]\Bigg)$$ from the rhs of \eqref{sum-prod} and by rearranging the terms according to the sets $\I_t$, we get:
\begin{multline}\label{hpI}
    \sum_{\bz=(\z^1,\ldots,\z^n)}\tP(\I,\bz)=\hP(\I)\cdot\prod_{t=1}^n\Bigg(\sum_{z^t=(z^t_1,\ldots,z_{k_t}^t)}\Pr[E_{t_1}]^{z^t_1-1}(1-\Pr[E_{t_1}])\cdots\\\Pr[E_{t_{k_t}-1}]^{z^t_{k_t-1}-1}(1-\Pr[E_{t_{k_t-1}}])\Bigg).
\end{multline}
The proof is now complete, by noticing that all the factors, except from $\hP(\I)$ in the rhs of \eqref{hpI} are equal to $1$.
\end{proof}
Thus, by \eqref{msr-multi}, \eqref{mv-msv} and \eqref{msvpn}, we get:\begin{equation}\label{msr-msv}
    \rP_n\leq\hP_n.
\end{equation}
Thus, what remains is to show that $\hP_n$ is inverse exponential to $n$. Towards this,  for $n\geq 1$ let
\begin{equation}\hP_{n, I} =  \sum_{\stackrel{\mathcal{I}: |\mathcal{I}| =n}{\mathcal{I}_1 =I }}\hP(\mathcal{I}),\end{equation}
where  $\mathcal{I}_1$ is its first term of  $\mathcal{I}$.

Observe now that, for any independent set $I$, $\hP_{1, I} = \prod_{j \in I} p_j$. Thus we obtain the following recursion:

\begin{equation} \label{recursion} 
\hP_{n+1,I}=
\begin{cases} 
\prod_{j \in I} p_j \left(\sum_{J: I \mbox{ covers } J } {\hP}_{n,J}\right) & \mbox{ if } n \geq  1, \\  
\prod_{j \in I}p_j & \mbox{ if } n=0. 
\end{cases}
\end{equation}
If the class of all non-empty independent sets is $\{I_1, \ldots, I_s\}$, following again the terminology of \cite{kolipakaszegedy2011}, we define the \emph{stable set matrix} $M$, as an $s\times s$ matrix, whose element in the $i$-th row and $j$-th column is $\prod_{j\in I}p_j$ if $I$ covers $J$ and $0$ otherwise. Furthermore, let $q_n=(\hP_{n,I_1},\ldots,\hP_{n,I_s})$. Easily, \eqref{recursion} is equivalent to: $$q_n=Mq_{n-1},$$ thus
\begin{equation}\label{recfinal}
q_n=M^{n-1} q_1. 
\end{equation}

\remove{{\color{red}\bf \large  Pls: (1) Adjust the notation form here and on (you should just end with the white box that signifies end of proof) --omit the computations about the thresholds, use the terminology of the final theorem as I wrote it above.  (2) Check everything very carefully and please send the final by tomorrow Tuesday morning 8 AM. I will not look at it.}}

Let $\|\cdot\|_1$ be the $1$-\emph{norm} defined on $\mathbb{R}^s$. It is known that any vector norm, and thus $1$-norm too, yields a norm for square matrices called the \emph{induced norm} \cite{horn1990matrix} as follows:
\begin{equation}\label{induced}
\|M\|_1:=\sup_{x\neq 0} \frac{\|Mx\|_1}{\|x\|_1}\geq\frac{\|Mq_1\|_1}{\|q_1\|_1}.
\end{equation}
By \eqref{recfinal} and \eqref{induced}, we have that:
\begin{equation}\label{norm}
    \|q_n\|_1=\|M^{n-1}q_1\|_1\leq \|M^{n-1}\|_1\cdot\|q_1\|_1.
\end{equation}
Note now that:
\begin{equation}\label{boundpn}
    \hP_n\leq\sum_{i=1}^s \hP_{n,I_i}=\|q_n\|_1=\|M^{n-1}\|_1\|q_1\|_1.
\end{equation}
Since $\|q_1\|_1$ is a constant, it suffices to show that $\|M^{n-1}\|_1$ is exponentially small in $n$. Let $\rho(M)$ be the \emph{spectral radius} of $M$ \cite{horn1990matrix}, that is: $$\rho(A):=\max\{|\lambda|\mid\lambda\text{ is an eigenvalue of } A\}.$$ By Gelfand's formula (see again \cite{horn1990matrix}) used for the induced matrix norm $\|\cdot\|_1$, we have that:
\begin{equation}\label{gelfand}
\rho(M)=\lim_{n\rightarrow\infty}\|M^n\|_1^{1/n}.
\end{equation}
Furthermore, in \cite{kolipakaszegedy2011} (Theorem $14$), it is proved that the following are \emph{equivalent}:
\begin{enumerate}
    \item For all $I\in I(G): \ q_I(G,\bar{p})> 0.$
    \item $\rho(M)<1$. 
\end{enumerate}
Using $(1\Rightarrow 2)$ we can select an $\epsilon>0$ such that $\rho(M)+\epsilon<1$. Then, by \eqref{gelfand}, we have that there exists a $n_0$ (depending only on $\epsilon, M$) such that, for $n\geq n_0$: $\|M^{n-1}\|_1 \leq (\rho(M)+\epsilon)^{n-1}$, which, together with \eqref{boundpn}, gives us that $\hP_n$ is exponentially small in $n$.

Thus, by the analysis above, we get that there is a constant $c<1$ (depending on $\|q_1\|, p$ and $\rho(M)+\epsilon$) such that $\rP_n\leq c^n$, for $n\geq n_0$ and by ignoring polynomial factors. This concludes the proof.

\end{proof}

\section*{Acknowledgment} We are truly grateful to Ioannis Giotis and Dimitrios Thilikos for their substantial contribution to earlier versions of this work (see \cite{Giotis2018AlternativePO} and \cite{DBLP:conf/colognetwente/GiotisKPT15}).

%\bibliography{lllgeneral}
%\bibliographystyle{plain}

\end{document}